\documentclass[11 pt,final]{article}
\usepackage{amsmath,amssymb,amsthm,amscd,amsfonts}
\usepackage[hmargin=1.25in, vmargin=1in]{geometry}
\usepackage[font=small,format=plain,labelfont=bf,up,textfont=it,up]{caption}
\usepackage{subcaption}
\usepackage{verbatim}
\usepackage{url}
\usepackage{calc}
\usepackage{graphicx}
\usepackage{multicol}
\usepackage{pstricks}
\usepackage{float}
\usepackage{overpic}
\usepackage[shortlabels]{enumitem}

\usepackage{tikz, tikz-cd}

\usepackage{listings}

\long\def\/*#1*/{}

\newcommand{\urlwofont}[1]{\urlstyle{same}\url{#1}}

\newcommand{\nc}{\newcommand}
\nc{\nt}{\newtheorem}
\nc{\dmo}{\DeclareMathOperator}

\theoremstyle{plain}
\newtheorem{theorem}{Theorem}[section]
\newtheorem{maintheorem}{Theorem}
\newtheorem{proposition}[theorem]{Proposition}
\newtheorem{lemma}[theorem]{Lemma}

\newtheorem{corollary}[theorem]{Corollary}

\theoremstyle{definition}

\theoremstyle{remark}

\newtheorem{example}[theorem]{Example}

\dmo{\SMod}{SMod}
\dmo{\PMod}{PMod}
\dmo{\SHomeo}{SHomeo}
\dmo{\SI}{\mathcal{SI}}
\dmo{\SSp}{SSp}
\dmo{\PSp}{PSp}

\newcommand\Z{\ensuremath{\mathbb{Z}}}

\newcommand\N{\ensuremath{\mathbb{N}}}

\nc{\p}[1]{\noindent {\bf #1.}}
\nc{\margin}[1]{\marginpar{\scriptsize #1}}

\nc{\PartialIBases}{\mathfrak{IB}}
\nc{\PartialIBasesEx}{\widehat{\mathfrak{IB}}}
\nc{\PartialBases}{\mathfrak{B}}
\nc{\Building}{\mathfrak{T}}
\nc{\height}{\ensuremath{\text{ht}}}
\nc{\Poset}{\mathfrak{P}}
\nc{\Field}{\mathbb{F}}
\nc{\Link}{\ensuremath{\text{Link}}}
\nc{\Star}{\ensuremath{\text{Star}}}

\nc{\SymTorelli}{\ensuremath{\mathcal{SI}}}
\nc{\BTorelli}{\ensuremath{\mathcal{BI}}}
\dmo{\Braid}{\ensuremath{B}}
\dmo{\PureBraid}{\ensuremath{PB}}
\nc{\Hyper}{\ensuremath{\iota}}

\nc{\BigFreeProd}{\mathop{\mbox{\Huge{$\ast$}}}}

\nc{\Quotient}{\ensuremath{\mathcal{Q}}}
\nc{\QuotientEx}{\ensuremath{\widehat{\mathcal{Q}}}}

\nc{\Presentation}[2]{\ensuremath{\text{$\langle #1$ $|$ $#2 \rangle$}}}
\nc{\SpGen}{\ensuremath{S_{\text{Sp}}}}
\nc{\SpRel}{\ensuremath{R_{\text{Sp}}}}
\nc{\QGen}{\ensuremath{S_{\mathcal{Q}}}}
\nc{\QRel}{\ensuremath{R_{\mathcal{Q}}}}
\nc{\PBs}{\ensuremath{T}}
\nc{\Qs}{\ensuremath{\overline{s}}}

\dmo{\PB}{PB}
\nc{\BIredg}{\mathcal{BI}_{2g+1}^{\text{red}}}
\nc{\BI}{\mathcal{BI}}
\dmo{\D}{D}
\dmo{\Stab}{Stab}
\dmo{\Surger}{Surger}
\nc{\I}{\mathcal{I}}

\nc{\spanmap}{span}

\nc{\genbygen}[2]{\premonoid{#1}{#2}}
\nc{\premonoid}[2]{#1 \circledcirc #2}
\nc{\monoid}[2]{#1 \odot #2}

\nc{\G}{\Gamma}
\nc{\raag}{A_\Gamma}
\nc{\racg}{W_\G}
\nc{\raagdelt}{A_\Delta}
\dmo{\Aut}{Aut}
\dmo{\Out}{Out}
\nc{\Autraag}{\Aut(\raag)}
\nc{\Outraag}{\Out(\raag)}
\nc{\diag}{D_\G}
\nc{\Autraagdelt}{\Aut(A_\Delta)}
\nc{\glk}{\GL(k,\mathbb{Z})}
\nc{\GLn}{\GL(n,\mathbb{Z})}
\nc{\GLnt}{\GL(n,\mathbb{Z} / 2)}
\nc{\glkdelt}{\GL(k |\Delta| ,\mathbb{Z})}
\nc{\gldelt}{\GL(|\Delta| ,\mathbb{Z})}
\nc{\zkdelt}{\mathbb{Z}^{k|\Delta|}}
\nc{\join}{\mathcal{J}}
\nc{\pc}{\mathrm{PC}}
\dmo{\lk}{lk}
\dmo{\st}{st}
\dmo{\Inn}{Inn}

\nc{\pia}{\Pi \mathrm{A}}
\nc{\piaG}{\pia_\G}
\nc{\ppia}{\mathrm{P \Pi A}}
\nc{\ppiaG}{\mathrm{P \Pi A}_\G}
\nc{\epia}{\mathrm{E \Pi A}}
\nc{\epiaG}{\mathrm{E \Pi A}_\G}
\nc{\ptor}{\mathcal{PI}}
\nc{\ptorG}{\mathcal{PI}_\G}
\nc{\CGi}{C_\G(\iota)}
\nc{\pbc}{\mathfrak{B}^\pi}
\dmo{\Cay}{Cay}
\dmo{\rev}{rev}
\nc{\Autfn}{\Aut(F_n)}
\dmo{\supp}{supp}
\dmo{\rk}{\mathrm{rk}}
\dmo{\PCT}{\mathrm{PCT}(\raag)}
\dmo{\PCTo}{\overline{\mathrm{PCT}}(\raag)}

\dmo{\HOutn}{\mathrm{HOut}(F_n)}
\dmo{\STn}{\mathcal{ST}(n)}
\dmo{\HMn}{\mathrm{HM}_n(\Z)}

\newcommand{\GL}{\text{GL}}

\newcommand{\sF}{\mathcal{F}}

\newcommand{\set}[1]{\left\{#1\right\}}

\renewcommand{\mod}[1]{\ (\text{mod } #1)}

\DeclareMathOperator{\minnorm}{min}
\DeclareMathOperator{\mult}{mult}

\title{Observed periodicity related to the four-strand Burau representation}
\author{Neil J. Fullarton and Richard Shadrach}

\begin{document}
 \maketitle
 \begin{abstract}
  A long-standing open problem is to determine for which values of $n$ the Burau representation $\Psi_n$ of the braid group $B_n$ is faithful. Following work of Moody, Long--Paton, and Bigelow, the remaining open case is $n=4$. One criterion states that $\Psi_n$ is unfaithful if and only if there exists a pair of arcs in the $n$-punctured disk $D_n$ such that a certain associated polynomial is zero. In this paper, we use a computer search to show that there is no such arc-pair in $D_4$ with 2000 or fewer intersections, thus certifying the faithfulness of $\Psi_4$ up to this point. We also investigate the structure of the set of arc-pair polynomials, observing a striking periodicity that holds between those that are, in some sense, `closest' to zero. This is the first instance known to the authors of a deeper analysis of this polynomial set.
 \end{abstract}

 \section{Introduction}
  Artin's braid group $B_n$ appears throughout mathematics in many guises, in large part due to its fundamental connection to the motion of sets of particles in the plane. The group $B_n$ enjoys many desirable properties, such as finite-presentability and linearity. The latter was established via the faithful \emph{Lawrence--Krammer representation} of $B_n$; however, a long-standing open problem is for which values of $n$ the related \emph{Burau representation} $\Psi_n$ of the braid group is faithful. It is known to be faithful for $n \leq 3$ and not faithful for $n \geq 5$. In this paper, we are concerned with the question of faithfulness when $n=4$, the remaining open case. A non-trivial element in the kernel of $\Psi_4$ would give a likely candidate for showing that the Jones polynomial cannot detect the unknot, as noted by Bigelow \cite{Big02}.

The Burau representation's lack of faithfulness for large $n$ was shown by Moody \cite{Moo91}, Long--Paton \cite{LP93}, and Bigelow \cite{Big99}, using related criteria involving combinatorics of pairs of arcs in the $n$-punctured disk $D_n$. To each pair of arcs in $D_n$ we associate a polynomial called its \emph{Burau polynomial}. One criterion states that $\Psi_n$ is not faithful if and only if there exists a pair of intersecting arcs in an explicit set $\mathcal{A}$ whose Burau polynomial is zero. Ordering arc-pairs by their geometric intersection number, the first main theorem of this paper certifies the faithfulness of the Burau representation $\Psi_4$ up to arc-pairs with intersection number 2000.

\begin{maintheorem}\label{the:certify}Any pair of arcs $(\alpha, \beta) \in \mathcal{A}$ with geometric intersection number at most 2000 has non-zero Burau polynomial.
\end{maintheorem} 

We prove Theorem~\ref{the:certify} by computer search, identifying then testing a finite yet sufficient collection of arc-pair representatives for each intersection number. We describe the arcs using weighted train tracks, with linear inequalities between the weights being used to `steer' the arcs around the disk $D_n$.

Computation of Burau polynomials has been carried out previously. Working with a different collection of arcs in $D_n$ referred to as `noodles' and `forks', Bigelow \cite{Big02} certified the faithfulness of $\Psi_4$ up to noodles and forks intersecting 2000 times. Translating this into the language used here, this corresponds to a set of arc-pairs intersecting at most 1000 times. The bound obtained in Theorem~\ref{the:certify} is thus of notable increase. Also, the collection of arcs in \cite{Big02} and those in Theorem~\ref{the:certify} are of entirely different types. While either collection may be used to certify faithfulness, we are unaware of any way to compare the two.  Our approach benefited from decreased computation time, due to several short-cuts we were able to make. These are discussed in detail in Section 4.

One by-product of the proof of Theorem~\ref{the:certify} is the means to list a slew of Burau polynomials. In doing so, a striking periodicity is observed, as we now describe. We define the \emph{norm} of a Burau polynomial to be the sum of the absolute values of its coefficients. For reasons apparent in the calculation of these polynomials, if seeking the zero polynomial, we need only check arc-pairs with even intersection number. The following theorem records a repetitive structure observed in the minimum norm for such intersection numbers.

\begin{maintheorem}\label{the:period}Let $2k \in [44, 500]$. If $2k$ is divisible by $6$, then the minimum norm over all Burau polynomials arising from an arc-pair with $2k$ geometric intersections is $10$. Otherwise, the minimum norm is $8$.\end{maintheorem}

The upper bound of 500 was imposed by limits on computation time; this periodicity almost certainly persists longer. Immediately, we observe that if $\Psi_4$ is not faithful, there must be some large intersection number at which the periodicity breaks.

We remark that the largest intersection numbers checked in Theorems \ref{the:certify} and \ref{the:period} are of different magnitudes. Both were pushed as high as possible, using hundreds of hours of computation time. However, the former only required testing whether the given polynomials equaled zero, whereas the latter demanded that we form a complete list of Burau polynomials to examine. For a given geometric intersection number, the latter case is more computationally laborious.

A strong relationship is observed between the arcs of minimum norm that appear in Theorem~\ref{the:period}. Every such arc falls into a family, where within a single family, the train track weights lie in an arithmetic progression and the ordered sets of coefficients of their Burau polynomials are all the same. Attempting to see the behavior of Burau polynomials across all arc-pairs, these `minimum norm families' obscure our view. By grouping all arcs into similar families, we are able to see through the minimum norms obtained for each intersection number and give numerical evidence that suggests $\Psi_4$ is faithful. This is discussed further in Section 5.

\textbf{Outline of paper.} In Section 2, we recall the definition of the Burau representation, and state criteria for its faithfulness. Section 3 explains how we use weighted train tracks to describe arcs in the disk $D_n$, and how we are able to reduce to only checking arc-pairs of a certain form. Section 4 details the algorithm we constructed, and in Section 5 we report on our observations and numerical evidence.

\textbf{Acknowledgements.} The authors would like to thank Stephen Bigelow, Dan Margalit, and Bal\'azs Strenner for helpful conversations. Part of this work was completed while the first author was in residence at the Mathematical Sciences Research Institute in Berkeley, California, during the Fall 2016 semester.
 \section{The Burau representation}
  Let $n\geq 2$. We begin this section by reviewing the definition of the Burau representation 
\[\Psi_n : B_n \to \mathrm{GL}_{n-1}(\mathbb{Z}[t^{\pm1}]).\] Then we discuss the criterion for faithfulness introduced by Moody and refined by Long--Paton and Bigelow.

\subsection{Burau via braids acting on cyclic covers}

One incarnation of the braid group $B_n$ is as the mapping class group of $D_n$ \cite{Bir74}, the disk with $n$ marked points in its interior, as depicted in Figure~\ref{disk} for $n=4$. Viewing the marked points as punctures, the fundamental group of $D_n$ is the free group $F_n$. With $p_0$ a basepoint on the boundary, a free basis is given by $X := \{ x_1, \dots , x_n \}$, where $x_i$ is the loop traveling around $p_i$ in the clockwise direction.

\begin{figure}[h]
\begin{center}
\begin{overpic}[width=2.7in]{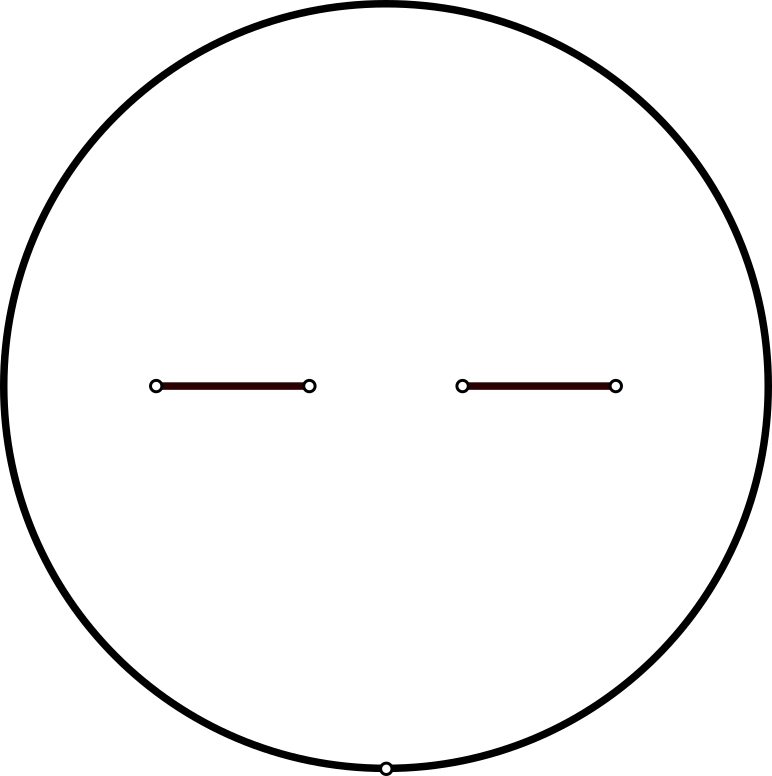}
\put(48,7){$p_0$}
\put(19,43){$p_1$}
\put(39,43){$p_2$}
\put(58,43){$p_3$}
\put(78,43){$p_4$}
\put(29,55){$\alpha$}
\put(68,55){$\beta$}
\end{overpic}
\caption{The $4$-punctured disk, $D_4$. Straight line arcs $\alpha$ and $\beta$ are depicted.}
\label{disk}
\end{center}
\end{figure}
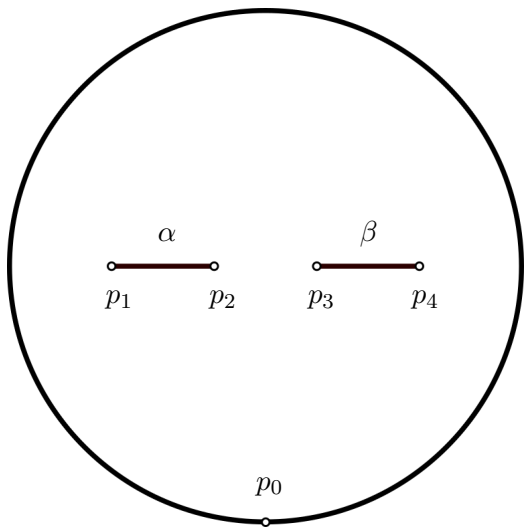

Consider the cyclic cover $\tilde D_n$ of $D_n$ associated with the kernel of the exponent sum map $F_n = \langle X \rangle \to \Z$. This cover is formed by taking a copy or `level' of $D_n$ for each integer, cutting along a straight-line arc from the boundary to each puncture in every level, and then gluing appropriately. The cover $\tilde D_n$ has the structure of a free $\mathbb{Z}[t^{\pm1}]$-module endowed upon its first (integral) homology group, with $t$ corresponding to a choice of generator of the cover's deck group, $\Z$. The action of $B_n$ on $D_n$ preserves the kernel of the exponent sum map, and thus lifts to $\tilde D_n$. The free $\mathbb{Z}[t^{\pm1}]$-module $H_1(\tilde D_n)$ has rank $n-1$, and the \emph{(reduced) Burau representation} $\Psi_n : B_n \to \mathrm{GL}_{n-1}(\mathbb{Z}[t^{\pm1}])$ assigns to each braid the $\Z[t^{\pm 1}]$-linear transformation of $H_1(\tilde D_n) \cong \mathbb{Z}[t^{\pm1}] ^ {(n-1)}$ that it induces.

\subsection{A criterion for faithfulness}

An \emph{arc of type $(i,j)$} in $D_n$ is the (oriented) image of a continuous embedding of the unit interval $[0,1]$ into $D_n$, whose endpoints are $p_i$ and $p_j$. Given arcs $\alpha$ and $\beta$ of some (perhaps distinct) type, the \emph{Burau polynomial} \cite{LP93} is defined to be
\[ \int_{\beta} \alpha := \sum_{j \in \mathbb{Z}} \hat \iota (t^j \cdot \tilde \alpha , \tilde \beta) t^j ,\] where $\tilde \alpha$ and $\tilde \beta$ are some choice of lifts of $\alpha$ and $\beta$ respectively, and $\hat \iota$ denotes algebraic intersection in $\tilde D_n$. Note that this polynomial is only well-defined up to multiplication by $t$, and that $\int_{\beta} \alpha$ and $\int_{\alpha} \beta$ are obtained from each other by replacing $t$ by $t^{-1}$ and negating coefficients. For a comprehensive discussion of calculating Burau polynomials in practice, we refer to Bigelow \cite{Big99}.

The following criteria, written as appears in \cite{Big99}, allow us to use Burau polynomials to investigate the faithfulness of the Burau representation.

\begin{theorem}[Bigelow \cite{Big99}] \label{the:criterion}Let $n \geq 3$. The following are equivalent:
\begin{enumerate}\item The Burau representation $\Psi_n$ is unfaithful.
\item There exist essentially intersecting arcs $\alpha$ and $\beta$ in $D_n$ of types $(1,2)$ and $(3,4)$ respectively, with zero Burau polynomial.
\item There exist essentially intersecting arcs $\alpha$ and $\beta$ in $D_n$ of types $(1,2)$ and $(0,3)$ respectively, with zero Burau polynomial. 
\end{enumerate}
\end{theorem}
The proof of Theorem~\ref{the:criterion} is given in Bigelow. While it does not explicitly show that the second statement implies the third, this may be established using techniques similar to those found elsewhere in the proof.

Variants of these criteria for faithfulness have been used to show that $\Psi_n$ is not faithful for large $n$. Moody \cite{Moo91} initially obtained a lack of faithfulness for $n \geq 9$, which was improved to $n \geq 6$ by Long--Paton \cite{LP93}. Most recently, Bigelow \cite{Big99} showed that $\Psi_5$ was not faithful. The remaining case to investigate is $n = 4$; our goal for the remainder of the paper is to do so.

To close this section, we give a proof that the Burau polynomial of an arc-pair is invariant under the action of the braid group. We will use this fact repeatedly to simplify the calculations we carry out.

\begin{proposition}Let $\phi \in B_n$ for $n \geq 2$. Then for any arc $\alpha$ of type $(i,j)$ and $\beta$ of type $(k, l)$ in $D_n$, we have \label{preserve}
\[ \int_{\beta} \alpha = \int_{\phi \cdot \beta} \phi \cdot \alpha . \]
\end{proposition}
\begin{proof}
We check that the proposition holds when $\phi$ is one the standard half-twists generating the braid group $B_n$ (i.e. the half-twist $\sigma$ interchanging adjacent punctures seen in Figure \ref{halftwist}).

Given any arcs $\alpha$ and $\beta$, we may assume that any points of intersection between them lie outside of the support of $\sigma$, by either isotoping the arcs or the half-twist. Locally, $\alpha$ and $\beta$ can be taken to look as in Figure \ref{pre}, with, perhaps, the segment of the red curve beginning and ending in a different position relative to the punctures.

\begin{figure}
\centering
\begin{subfigure}[t]{0.4\textwidth}\centering
\includegraphics[width=1.6in]{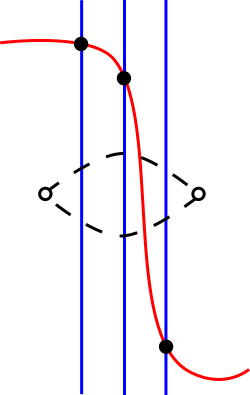}  
\caption{The half-twist $\sigma$ is supported on a regular neighborhood of the two dashed arcs.}
\label{pre}
\end{subfigure}
\quad
\begin{subfigure}[t]{0.4\textwidth}\centering
\includegraphics[width=2in]{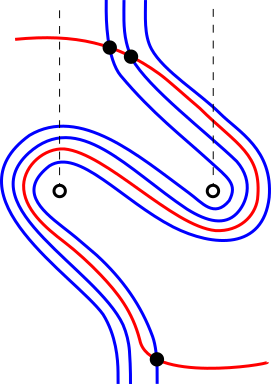} 
\caption{The dashed vertical lines indicate when the arcs' lifts change levels in the cover $\tilde D_n$.}
\label{post}
\end{subfigure}
\caption{The effect of $\sigma$ on the arc-pair $(\alpha, \beta)$. Subfigure (a) depicts a generic local picture of intersections between the pair. Subfigure (b) shows the result of applying $\sigma$.} \label{halftwist}
\end{figure}

When the half-twist $\sigma$ is carried out, the arc segments weave first around one puncture, then the other. Passing to the cover $\tilde D_n$, this corresponds to the arcs' lifts first climbing one level, traveling to the other puncture, then descending to the original level. Since no intersections between $\alpha$ and $\beta$ occur during this detour, it has no effect upon our calculation of the Burau polynomial of the pair. Thus, the half-twist $\sigma$ leaves the Burau polynomial unchanged.
\end{proof}

 \section{Chasing arcs on train tracks}
  Our ultimate aim is to use Theorem~\ref{the:criterion} to certify the faithfulness of $\Psi_4$ up to some large number of intersections between the pairs of arcs appearing in the theorem. Given that the existence of a pair $(\alpha, \beta)$ with $\beta$ of type $(3,4)$ and zero Burau polynomial is equivalent to the existence of some pair $(\alpha', \beta')$ with $\beta '$ of type $(0,3)$ and zero Burau polynomial, we limit our search to those of the former type. Thus we will show that, up to some large bound $N$, any arcs $\alpha$ of type $(1,2)$ and $\beta$ of type $(3,4)$ intersecting $N$ times or fewer have non-zero Burau polynomial.
 
In practice, the techniques we use could be implemented to approach pairs of the latter type. Bigelow carried out a computer search on this latter type using different methods \cite{Big02}. There, his `noodles' and `forks' are equivalent to arc-pairs of types $(1,2)$ and $(0,3)$: the fork's tine is an arc of type $(1,2)$, while a noodle $N$ gives rise to an arc of type $(0,i)$, by joining $p_0$ to the puncture $p_i$ separated from the rest by $N$. We note that passing from a noodle to its arc halves the intersection number with the fork's tine, resulting in the halving from 2000 to 1000 stated in this paper's introduction.

Since our focus will be on arc pairs $(\alpha, \beta)$ of types $(1,2)$ and $(3,4)$, we will call any such pair a \emph{Burau pair}.
 
\subsection{Train tracks}

Given an arc in $D_n$, we encode its path using a finite data set, via a train track. For our purposes, a \emph{train track} $\tau$ in the punctured disk $D_n$ is an embedded trivalent graph $\Gamma$, where all edges incident at a vertex $v$ have a well-defined, common tangent line at $v$. It is also required that each $v$ has exactly two edges entering from one direction. We refer to the edges of $\Gamma$ as \emph{rails} and the vertices as \emph{switches}. Figure~\ref{track} displays an example of a train track in $D_4$. We note that our definition is weaker than the usual definition of train track, in that the complement of a train track may have a disk as a connected component.

Let $\tau$ be some train track in $D_n$. A simple closed curve $\gamma$ in $D_n$ is \emph{carried by $\tau$} if $\gamma$ is homotopic to a regular smooth path in $\tau$. The number of points in $\gamma$ that are homotoped to the same point of a rail $r$ is well-defined over all such homotopies, called the \emph{weight} of $\gamma$ on $r$. More loosely, the weight of $\gamma$ on $r$ is the number of segments near $r$ after $\gamma$ has been homotoped close to $\tau$.

Given a switch $v$, let $w_1$, $w_2$ and $w_3$ denote the weights of a curve $\gamma$, carried by $\tau$, on the rails $r_1$, $r_2$ and $r_3$ of $v$, respectively. If $r_1$ and $r_2$ approach $v$ from the same direction, the \emph{switch condition}
\[w_1+w_2 = w_3 \]
must hold between the weights. Any set of weights on the rails of $\Gamma$ that satisfies every switch condition gives rise to some simple, closed multi-curve (that is, some collection of pairwise disjoint simple closed curves).

We will be concerned with describing arcs rather than simple closed curves. To pass from the former to the latter, we take the boundary of a regular neighborhood of the arc. We discuss how to reverse this procedure in Section 4, giving necessary and sufficient conditions that a given multi-curve contains a curve that arises from an arc in this fashion. In practice, we will discard any proper multi-curve (i.e. a multi-curve with at least two components). The procedure we use is discussed at the end of Section 4.

\subsection{A universal track}

Due to our weakened definition of a train track in $D_n$, we may carry all simple closed curves in $D_4$ on a single track. This track is seen in Figure~\ref{track}, and we denote it by $\mu$.
\begin{figure}[h]
\begin{center}
\begin{overpic}[width=2.7in]{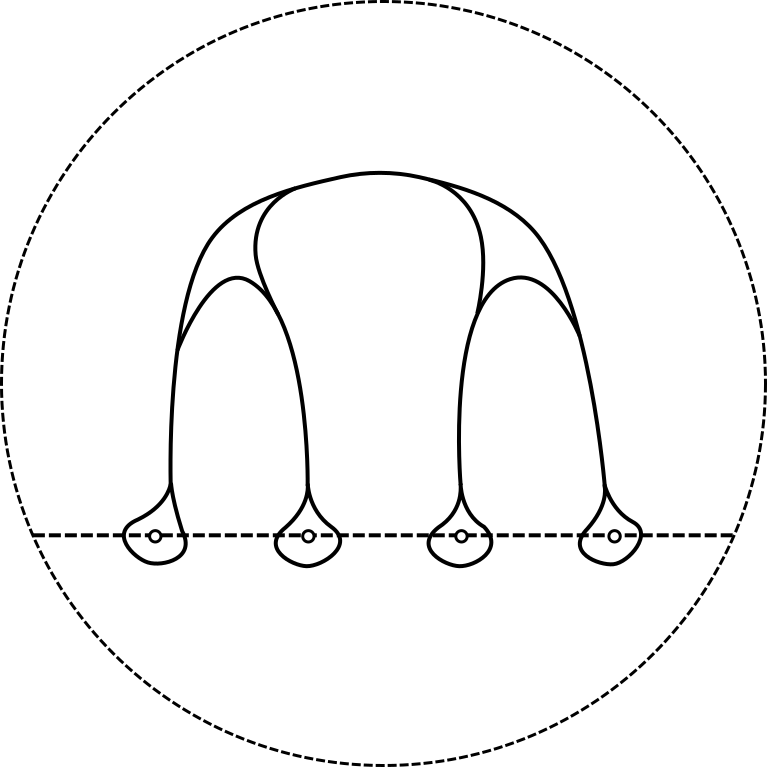}
\end{overpic}
\caption{The universal train track $\mu$, together with an ideal decomposition of $D_4$ into two hexagons.}
\label{track}
\end{center}
\end{figure}
\begin{proposition}\label{pro:universal}Any simple closed curve $\gamma$ in $D_4$ is carried by the train track $\mu$.
\end{proposition} 
\begin{proof}
We decompose $D_4$ as the (ideal) polygonal complex seen in Figure~\ref{track}. Up to isotopy, any simple closed curve $\gamma$ may be described via a sequence of line segments, each of which enters then leaves some polygon in this decomposition of $D_4$. We may assume that no segment enters and leaves a polygon on the same 1-cell, and that no segment intersects a 1-cell lying on the boundary of $D_4$. Each segment, up to isotopy, is uniquely determined by which two 1-cells it meets.

To show that any $\gamma$ is carried by the track $\mu$, it thus suffices to show that the finitely many possible line segments in the decomposition's polygons all may be homotoped so that they lie on the rails of $\tau$. Figure~\ref{track} may be used to verify this.
\end{proof}

\subsection{Standard forms for arcs}
\label{subsect:standardForms}
A priori, to use Theorem~\ref{the:criterion} to investigate the faithfulness of $\Psi_4$, there are infinitely many pairs of arcs with a fixed intersection number for which to calculate the Burau polynomial. In this section, we describe how to reduce to checking only finitely many pairs for each intersection number.

Our first reduction follows from the definition of the Burau polynomial, since each intersection between $\alpha$ and $\beta$ contributes $\pm t^i$ to their polynomial.

\textbf{Reduction I:} We need only calculate the Burau polynomial of arc pairs $(\alpha, \beta)$ whose (geometric) intersection number is even and positive.

Our subsequent reductions utilize the action of the braid group $B_4$ on $D_4$. As discussed by Bigelow, any Burau pair $(\alpha ', \beta ')$ yields an explicit, non-trivial element in the kernel of $\Psi_n$. Letting $T_{\gamma}$ denote the half-twist about an arc $\gamma$ in $D_4$, this explicit element is the commutator $[T_{\alpha '}, T_{\beta '}]$. Now, observe that any arc of type $(1,2)$ may be carried to the straight line arc $\alpha$ seen in Figure~\ref{disk} by acting by an appropriate member $\sigma$ of the braid group $B_4$. It follows from Proposition~\ref{preserve} that $(\alpha ', \beta ')$ is a Burau pair if and only if $(\alpha, \sigma \cdot \beta ')$ is a Burau pair. This leads to our second reduction.

\emph{Our subsequent reductions utilize the action of the braid group $B_4$ on $D_4$. Observe that any arc of type $(1,2)$ may be carried to the straight line arc $\alpha$ seen in Figure~\ref{disk} by acting by an appropriate member $\sigma$ of the braid group $B_4$. It follows from Proposition~\ref{preserve} that $(\alpha ', \beta ')$ has the same Burau polynomial as $(\alpha, \sigma \cdot \beta ')$. This leads to our second reduction.}

\textbf{Reduction II:} We need only calculate the Burau polynomial of arc pairs $(\alpha, \beta)$ where $\alpha$ is the straight line arc joining $p_1$ and $p_2$.

To perform our final reduction, let $N$ denote a regular neighborhood of the arc $\alpha$, and let $B_4(\partial N)$ denote the stabilizer in $B_4$ of the isotopy class of the boundary $\partial N$ of $N$. By letting the marked points $p_3$ and $p_4$ wander in the twice-punctured annulus $D_4 \setminus N$, we see that any arc $\beta '$ of type $(3,4)$ has a (perhaps different) representative whose initial and terminal segments are one of the possibilities seen in Figures \ref{initial} and \ref{terminal}, respectively. This is because such a wandering may be achieved by acting by a braid in $B_4(\partial N)$. 

\begin{figure}
\centering
\begin{subfigure}[t]{0.4\textwidth}\centering
\includegraphics[width=2in]{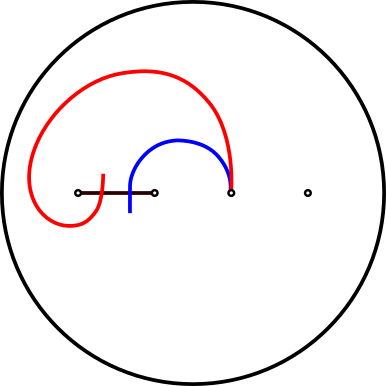}  
\caption{}
\label{initial}
\end{subfigure}
\quad
\begin{subfigure}[t]{0.4\textwidth}\centering
\includegraphics[width=2in]{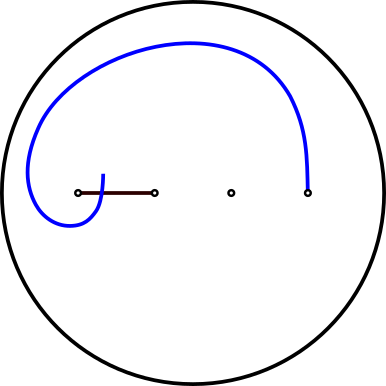} 
\caption{}
\label{terminal}
\end{subfigure}
\caption{Initial (a) and terminal (b) segments of the arc $\beta$ after our reduction process. We orient $\beta$ so that it begins at $p_3$ and ends at $p_4$.} \label{segments}
\end{figure}

Suppose two arcs $\beta_1$ and $\beta_2$ have as initial and terminal segments one of the choices shown in Figure~\ref{segments}. If $\beta_1$ and $\beta_2$ lie in the same orbit under the action of $B_4(\partial N)$, then they are isotopic rel. endpoints. We see this as follows. Certainly, if $\beta_1$ and $\beta_2$ belong to the same $B_4(\partial N)$-orbit, they must have the same choice of initial and terminal segment, since $B_4(\partial N)$ acts on the annulus $D_4 \setminus N$, preserving whether the segments both intersect $\alpha$ on the same or opposing sides. Assume $\beta_1$ and $\beta_2$ agree on their initial and terminal segments and belong to the same $B_4(\partial N)$-orbit. Then any braid $\sigma \in B_4(\partial N)$ mapping $\beta_1$ to $\beta_2$ must stabilize the isotopy class of a regular neighborhood of the union of $\alpha$ together with the initial and terminal segments of $\beta_1$. Any such $\sigma$ must be a power of the Dehn twist about a curve isotopic to the boundary of $D_4$, and hence we have $\beta_1 = \beta_2$. We are thus led to make the following reduction.

\textbf{Reduction III:} We need only calculate the Burau polynomial of arc pairs $(\alpha, \beta)$ where $\beta$ has initial and terminal segments as seen in Figure \ref{segments}.

To conclude this section, we summarize our reduction process. The above discussion allows us to only consider arc pairs $(\alpha, \beta)$ that essentially intersect a positive, even number of times, such that $\alpha$ is the straight-line arc from $p_1$ to $p_2$ and $\beta$ has one of the initial and terminal segment choices displayed in Figure~\ref{segments}. No such arcs forming a pair with $\alpha$ may be carried to one another via an element of $B_4$ without changing the isotopy class of the fixed arc $\alpha$.

Note that, given our reductions, there are finitely many arcs $\beta$ that intersect $\alpha$ some fixed number of times. We remark that such an easily described set of reductions is enjoyed only by the four-punctured disk $D_4$: for larger $n$, the presence of punctures without arcs attached make the situation inside $D_n$ more challenging to report.

 \section{Implementation}
  \newcommand{\indentedlist}{\setlength\itemindent{0.5cm}}
\newcommand{\specialcell}[2][c]{\begin{tabular}[#1]{@{}c@{}}#2\end{tabular}}

Let $\beta$ be an arc belonging to a Burau pair $(\alpha, \beta)$ subject to the reductions of Section \ref{subsect:standardForms}. The boundary of a regular neighborhood of $\beta$ is a loop $\gamma$ supported on $\mu$, and thus we can get a set of weights for $\gamma$. In this section we go the opposite direction, finding necessary and sufficient conditions that a set of weights produces a multicurve $\gamma$, one component of which is a regular neighborhood of an arc $\beta$ belonging to a Burau pair $(\alpha, \beta)$ subject to the reductions of Section \ref{subsect:standardForms}. Not being able to give conditions on the weights to prevent proper multicurves from arising, we instead detect proper multicurves by performing a preliminary computation on a given weight set. Details on this as well as our algorithm are given at the end of this section.

\begin{figure}[h]
 \begin{center}
  \begin{overpic}[width=0.8\textwidth]{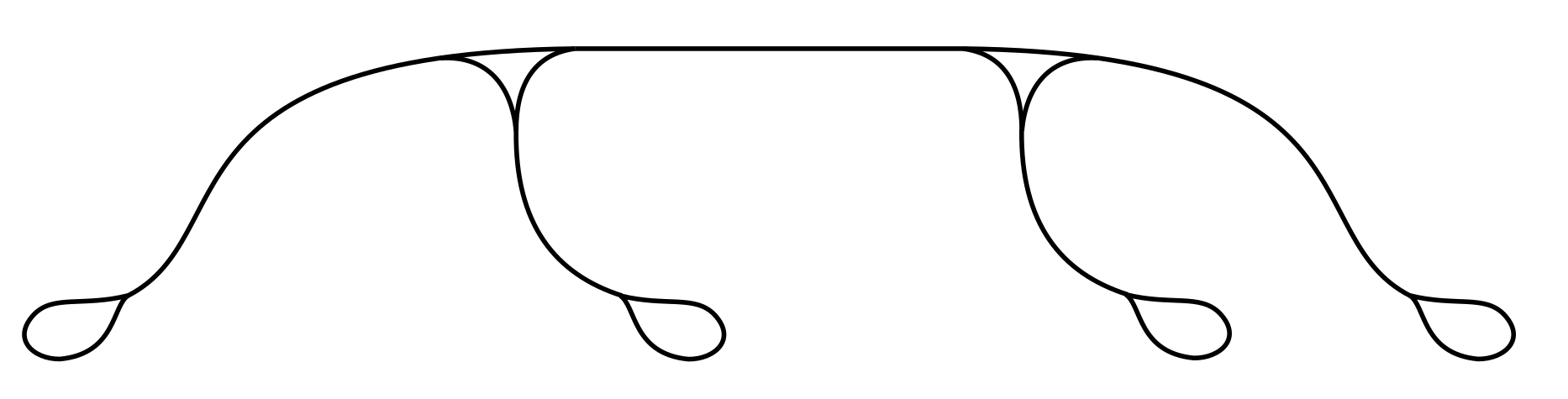}
   \put(0,1){$r_0$}
   \put(45,1){$r_1$}
   \put(78,1){$r_2$}
   \put(96,1){$r_3$}
   \put(8,13){$r_4$}
   \put(29,13){$r_5$}
   \put(61,13){$r_6$}
   \put(80,13){$r_7$}
   \put(27,19){$r_8$}
   \put(35,19){$r_9$}
   \put(30,24){$r_{10}$}
   \put(59,19){$r_{11}$}
   \put(67,19){$r_{12}$}
   \put(63,24){$r_{13}$}
   \put(45,24){$r_{14}$}
  \end{overpic}
  \caption{The train track $\mu$.}
  \label{fig:implementationTrainTrack}
 \end{center}
\end{figure}
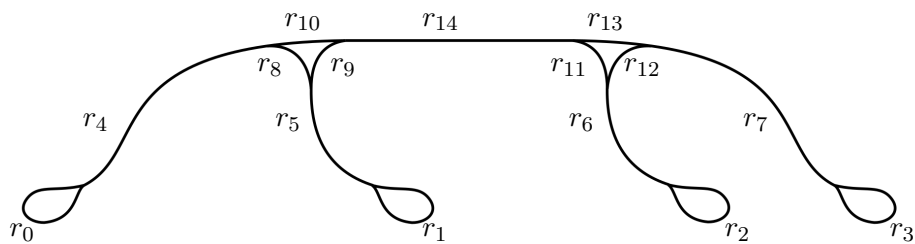

With $\N$ the nonnegative integers, let $(w_i)_{i=0}^{14} \in \N^{15}$ where $w_i$ is the weight of the $i$th rail of $\mu$ as labeled in Figure \ref{fig:implementationTrainTrack}. The switch conditions are expressed by the following equations.
\begin{align*}
 w_4 &= 2w_0 & w_8 &= w_0 + w_1 - \frac{w_{14}}{2} & w_{11} &= w_2 - w_3 + \frac{w_{14}}{2}\\
 w_5 &= 2w_1 & w_9 &= -w_0 + w_1 + \frac{w_{14}}{2} & w_{12} &= w_2 + w_3 - \frac{w_{14}}{2}\\
 w_6 &= 2w_2 & w_{10} &= w_0 - w_1 + \frac{w_{14}}{2} & w_{13} &= -w_2 + w_3 + \frac{w_{14}}{2}\\
 w_7 &= 2w_3
\end{align*}
As every weight is freely determined by $w_0$, $w_1$, $w_2$, $w_3$, and $w_{14}$, 
we focus our attention on the tuple $(w_0, w_1, w_2, w_3, w_{14})$.

\textbf{Divisibility:} The arc $\beta$ starts on rail $6$ and ends on rail $7$, so weights $w_0$ and $w_1$ are even and weights $w_2$ and $w_3$ are odd. 
\begin{center}
 \begin{tabular}{clcclcclccl}
  D1. & $2 \mid w_0$ & \qquad \qquad & D2. & $2 \nmid w_1$ & \qquad \qquad & D3. & $2 \mid w_2$ && D4. & $2 \nmid w_3$
 \end{tabular}
\end{center}

\textbf{Nonnegativity:} The weights $w_8,\ w_9,\ w_{10},\ w_{11},\ w_{12},$ and $w_{13}$ must be nonnegative.
\begin{center}
 \begin{tabular}{clcclccl}
  N1. & $\frac{w_{14}}{2} \leq w_0 + w_1$ & \qquad \qquad & N3. & $w_1 \leq w_0 + \frac{w_{14}}{2}$ & \qquad \qquad & N5. & $\frac{w_{14}}{2} \leq w_2 + w_3$\\
  N2. & $w_0 \leq w_1 + \frac{w_{14}}{2}$ && N4. & $w_3 \leq w_2 + \frac{w_{14}}{2}$ && N6. & $w_2 \leq w_3 + \frac{w_{14}}{2}$
 \end{tabular}
\end{center}

In the next two groups of conditions, we ``steer'' the initial and terminal segments as dictated by the reductions in Section \ref{subsect:standardForms}. Consider the track segment seen in Figure \ref{fig:TrackSegment}. Which rail the arc follows is determined by the number of lines $\ell$ that are to the left of the arc as it approaches a switch. Beginning at the switch of $r_0$ and $r_1$ traveling left, $\ell = w_0$. The arc is steered left if and only if $w_0 < w_2$. Assuming this is the case, after passing through the switch of $r_2$ and $r_4$ we still have $\ell = w_0$. On the other hand if the arc is steered right, then after passing through the switch of $r_3$ and $r_4$ we now have $\ell = w_0 - w_2 + w_4$. As the notions ``initial'' and ``terminal'' are merely a choice of orientation, these remarks equally apply to the terminal segment of the arc.

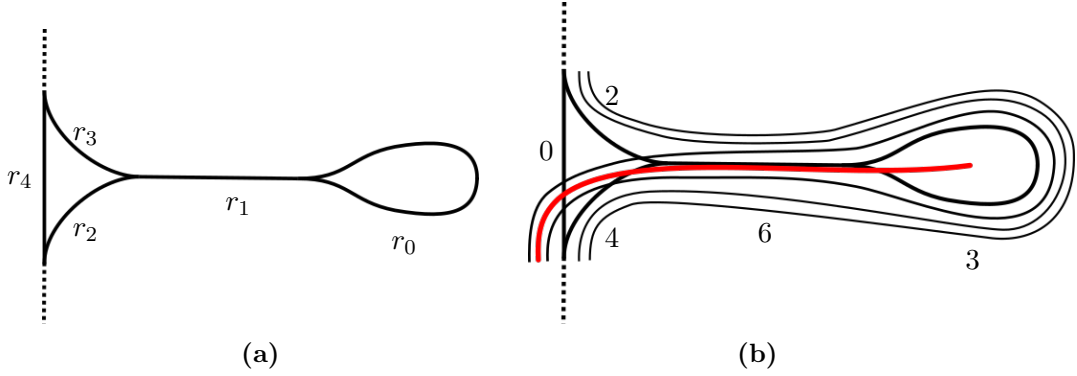
\begin{figure}
 \centering
 \begin{subfigure}[t]{0.4\textwidth}
  \centering
  \begin{overpic}[scale=0.45]{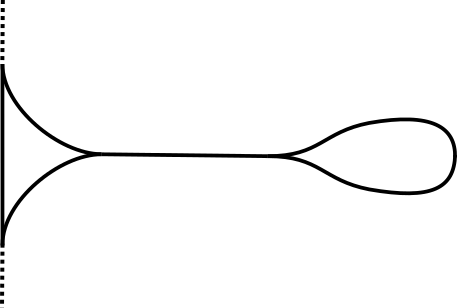}
   \put(80,17){$r_0$}
   \put(42,26){$r_1$}
   \put(7,20){$r_2$}
   \put(7,42){$r_3$}
   \put(-8,32){$r_4$}
  \end{overpic}
  \caption{}
 \end{subfigure}
 \quad
 \begin{subfigure}[t]{0.4\textwidth}
  \centering
  \begin{overpic}[scale=0.45]{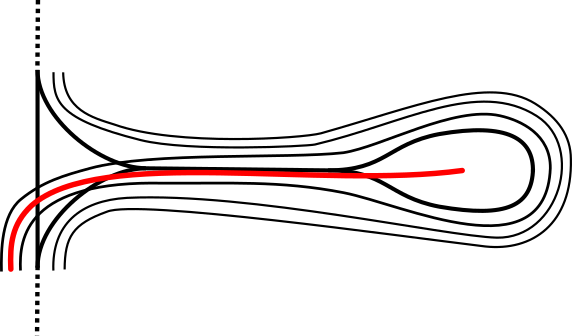}
   \put(80,10){$3$}
   \put(42,15){$6$}
   \put(14,14){$4$}
   \put(14,40){$2$}
   \put(2,30){$0$}
  \end{overpic}
  \caption{}
 \end{subfigure}
 \caption{A part of a train track where an arc is steered left.}
 \label{fig:TrackSegment}
\end{figure}

\textbf{Initial segment:} The arc $\beta$ starts on rail $6$ and then travels on rail $11$, meaning $w_2 < w_{11}$. After this, $\beta$ intersects the arc $\alpha$ from the top along rail $0$ or $1$, or from the bottom along rail $0$, meaning $w_9 < w_2$ or $w_1 < w_2$. In the case where $\beta$ intersects $\alpha$ from the bottom while traveling along rail $0$, we must enforce that $\beta$ does not then undo this intersection by traveling along rail $8$ and intersecting $\alpha$ from the top, as seen in Figure \ref{fig:badInitial}.

\begin{figure}[h]
 \begin{center}
  \includegraphics[scale=0.6]{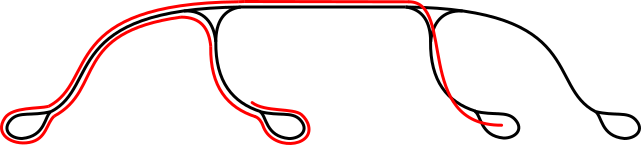}
  \caption{Initial intersection that can be homotoped off.}
  \label{fig:badInitial}
 \end{center}
\end{figure}

Therefore, if $w_9 < w_2$ and $w_0 < w_2 - w_9 + w_8$, then $w_2 - w_9 + w_8 < w_{10}$ or $w_2 + w_8 - w_{10} < w_1$. Using the switch conditions, we translate these into conditions on $w_0$, $w_1$, $w_2$, $w_3$, and $w_{14}$.
\begin{enumerate}[\indentedlist {I}1.]
 \indentedlist
 \item 
  $w_3 < \frac{w_{14}}{2}$
 \item
  $w_1 + \frac{w_{14}}{2}< w_0 + w_2$ \quad or \quad $w_1 < w_2$
 \item 
  $w_0 + w_2 < w_1 + \frac{w_{14}}{2}$ \quad or \quad $w_0 + w_2 < w_{14}$ 
\quad or \quad $w_0 + w_1 + w_2 < 3\frac{w_{14}}{2}$ \quad or \quad $w_1 + w_2 
< w_{14}$
\end{enumerate}

\textbf{Terminal segment:} The terminal segment of $\beta$ is seen in Figure \ref{fig:terminalPath}.
\begin{figure}
 \begin{center}
  \includegraphics[scale=0.6]{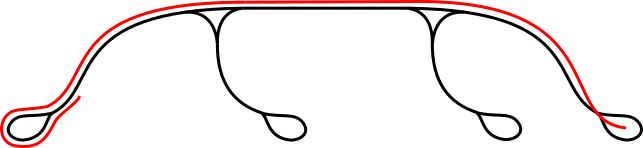}
  \caption{Terminal segment of an arc.}
  \label{fig:terminalPath}
 \end{center}
\end{figure}
This translates into the conditions $w_{12} < w_3$, $w_9 < w_3 - w_{12} + w_{11}$, and $w_0 < w_3 - w_{12} + w_{11} - w_9 + w_8$. Immediately prior to taking this path, the arc cannot travel on rail $1$ intersecting $\alpha$ from below. This gives the fourth condition
$$w_3 - w_{12} + w_{11} - w_9 + w_8 < w_{10} \quad \text{ or } \quad w_3 - 
w_{12} + w_{11} + w_8 - w_{10} < w_1.$$
Again using the switch conditions, we express these in terms of $w_0$, $w_1$, $w_2$, $w_3$, and $w_{14}$.
\begin{center}
 \begin{tabular}{clccl}
  T1. & $w_2 < \frac{w_{14}}{2}$ & \qquad \qquad & T3. & $w_3 < w_0$\\
  T2. & $w_1 + w_3 < w_0 + \frac{w_{14}}{2}$ && T4. & $w_0 + w_1 < w_3 + \frac{w_{14}}{2}$ \quad or \quad $w_1 < w_3$
 \end{tabular}
\end{center}

In order to prove that the conditions we have given are sufficient, we will use the following lemma.

\begin{lemma}[\cite{FM12}, Proposition 1.7]
 \label{lem:diagon}
 Two transverse simple closed curves $\alpha$ and $\beta$ in a surface $S$ cannot be homotoped to have fewer intersections if and only if they do not form a ``bigon'', that is, an embedded disk in $S$ whose boundary is the union of an arc of $\alpha$ and an arc of $\beta$.
\end{lemma}

\begin{proposition}
 \label{prop:sufficiency}
 The conditions D1-D4, N1-N6, I1-I3, and T1-T4 are necessary and sufficient for the weights $(w_i)_{i=0}^{14}$ to give rise to a multicurve $\gamma$, one component of which is a regular neighborhood of an arc $\beta$ belonging to a Burau pair $(\alpha, \beta)$ subject to the reductions of Section \ref{subsect:standardForms}.
\end{proposition}
\begin{proof}
 The necessity is clear. To show the sufficiency of the conditions, we must show that the initial and terminal intersections are essential. We prove a stronger statement, that the only nonessential intersection between $\alpha$ and $\beta$ supported on $\mu$ is of the form seen in Figure \ref{fig:badInitial}. Let $\alpha$ and $\beta$ intersect at $i_1$. Continuing to follow $\beta$, assuming it is not of the form in Figure \ref{fig:badInitial}, it must wind around a puncture $p_i$ before intersecting $\alpha$ again. Denote the next intersection of $\beta$ with $\alpha$ by $i_2$, and consider the loop $\eta$ formed by $\beta$ from $i_1$ to $i_2$ and $\alpha$ from $i_2$ to $i_1$. As $\eta$ is supported on the track, there is a corresponding set of weights $(v_i)_{i=0}^{14} \in \N^{15}$. The puncture $p_i$ will be interior to $\eta$ if and only if the corresponding weight $v_{i-1}$ is odd. If $v_0$, $v_1$, $v_2$, and $v_3$ are even, then by the switch conditions all the weights are even, implying that $\eta$ is a proper multicurve. Therefore one weight must be odd, so some $p_i$ is interior to $\eta$, and the result now follows from Lemma \ref{lem:diagon}.
\end{proof}

\begin{corollary}
 \label{cor:interesectionCount}
 Let $(w_i)_{i=0}^{14}$ be a set of weights as in Proposition \ref{prop:sufficiency} giving rise to a loop (and not a proper multicurve). Denote by $\beta$ the resulting arc. Then the number of essential intersections of $\alpha$ with $\beta$ is
  $$\frac{w_0}{2} + \frac{w_1}{2} - \min(w_1, w_8).$$
\end{corollary}

There are two ways we can prevent certain proper multicurves from arising. The first is to insist that $\gcd(w_0, w_1, w_2, w_3) = 1$, as otherwise all weights will share a non-trivial common factor and give rise to multiple copies of a single, simple closed curve. The second is to insist that some weight is zero, as otherwise the resulting multicurve will have a loop encompassing the entire train track and thus be a proper multicurve. The paths of the initial and terminal segments of the corresponding arc imply that all weights except for $w_1$, $w_8$, $w_9$, and $w_{12}$ are non-zero. Since $w_1$ being zero implies that $w_8$ and $w_9$ are zero, we have three cases.
\begin{enumerate}[(I)]
 \indentedlist
 \item
  $w_8 = 0$, meaning $\frac{w_{14}}{2} = w_0 + w_1$
 \item
  $w_9 = 0$, meaning $\frac{w_{14}}{2} = w_0 - w_1$
 \item
  $w_{12} = 0$, meaning $\frac{w_{14}}{2} = w_2 + w_3$
\end{enumerate}
All three cases can occur simultaneously. In each case, many of the above conditions become satisfied or simplified. The conditions D1-D4 remain in all cases, and the other conditions become the following.
\begin{center}
 Case I.
 
 N1. $w_0 + w_1 \leq w_2 + w_3$ \qquad I2. $w_1 < w_2$ \qquad T1. $w_2 < w_0 + w_1$ \qquad  T3. $w_3 < w_0$
\end{center}

\begin{center}
 Case II.
 
 \begin{tabular}{clccl}
  N5. & $w_0 \leq w_1 + w_2 + w_3$ & \qquad \qquad & I3. & $w_1 + w_3 < w_0$\\
  I1. & $w_0 < w_2 + 2w_1$ \quad or \quad $4w_1 + w_2 < 2w_0$  && T1. & $w_1 + w_2 < w_0$\\
  I2. & $w_1 < w_2$ && T4. & $w_1 < w_3$\\
 \end{tabular}
\end{center}

\begin{center}
 Case III.
 
 \begin{tabular}{clccl}
  N1. & $w_2 + w_3 \leq w_0 + w_1$ & \qquad & T2. & $w_1 < w_2 + w_0$\\
  N2. & $w_0 \leq w_1 + w_2 + w_3$ && T3. & $w_3 < w_0$\\
  I2. & $w_1 + w_3 < w_0$ \quad or \quad $w_1 < w_2$ && T4. & $w_0 + w_1 < w_2 + 2w_3$ \quad or \quad $w_1 < w_3$\\
  I3. & \multicolumn{4}{l}{$w_0 < w_1 + w_3$ \quad or \quad $w_2 + 2w_3 < w_0$ \quad or \quad $w_0 + w_1 < 2w_2 + 3w_3$}\\
 \end{tabular}
\end{center}

When traveling on rail $14$ either moving left or right, the conditions on the initial and terminal segments imply that there are only a finite number of possible paths one may take before returning to rail $14$. These paths are enumerated in Figure \ref{fig:paths}.
\begin{figure}
 \centering
 \begin{subfigure}{0.4\textwidth}
  \centering
  \includegraphics[scale=0.4]{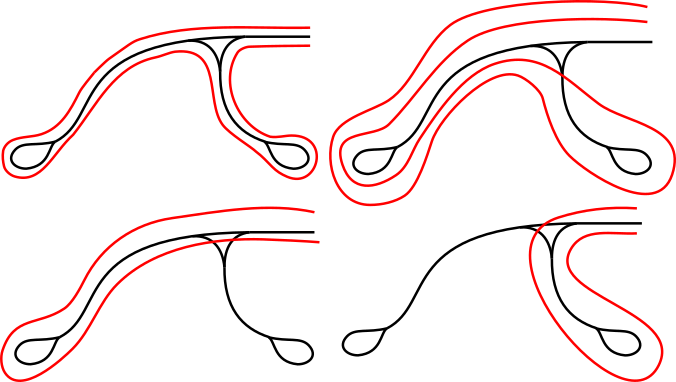}
  \captionof{figure}{Left paths.}
  \label{fig:leftPaths}
 \end{subfigure}
 \qquad \qquad
 \begin{subfigure}{0.4\textwidth}
  \centering
  \includegraphics[scale=0.4]{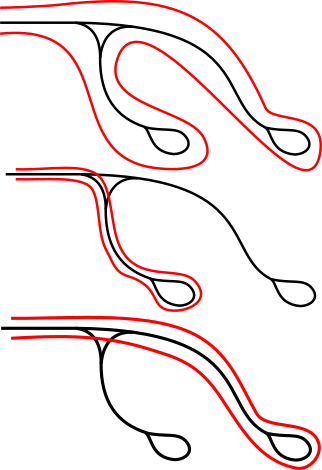}
  \captionof{figure}{Right paths.}
  \label{fig:rightPaths}
 \end{subfigure}
 \caption{Enumeration of paths.}
 \label{fig:paths}
\end{figure}
Let $\ell$ denote the number of lines to the left of the arc as it travels on rail 14. Which of the enumerated paths is taken, along with the direction this path is traversed, is determined by comparing $\ell$ to the train track weights. When traveling to the left, the path and direction are determined by the position of the first number that $\ell$ is smaller than in the following list.
 $$\min(w_1, w_8, w_9),\ \min(w_1, w_9),\ w_9,\ w_0 + w_9 - w_8,\ w_{10} + w_9 - w_8,\ w_1 + w_{10} - w_8,$$
 $$w_{14} - w_1,\ \min(w_{10}, w_0 + w_{10} - w_8, 2w_{10} - w_8)$$
Likewise when traveling to the right, the path and direction are determined by the position of the first number that $\ell$ is smaller than in the following list.
 $$\min(w_3, w_{12}, w_{13}),\ \min(w_3, w_{13}),\ w_{13},\ w_2 + w_{13} - w_{12},\ w_{11} + w_{13} - w_{12}$$
Note that for a given set of weights, it may be that either of the lists are not ascending. This can occur when an arc never takes one of the paths enumerated in Figure \ref{fig:leftPaths} and Figure \ref{fig:rightPaths}. 

To compute the resulting Burau polynomial, our algorithm begins with $\ell = w_2$ and the Burau polynomial being $0$. We then alternate traveling left and right on rail $14$. Each time a path is taken, we update $\ell$ and the Burau polynomial accordingly. We end the computation when the arc reaches rail $7$ with $\ell = w_3$. While traveling right on rail 14, this is precisely when we have either $\ell = w_3$ and $\ell < \ w_{13}$, or $\ell$ is greater than all the numbers in the list above and $\ell = 3w_3 + w_{14}$. Along with being computationally efficient, this allows us to concisely describe the path of an arc. Each time we return to rail 14, using the appropriate list we record the position of the first number that $\ell$ is smaller than, $1-8$ and $1-6$ respectively. We call the resulting list \emph{the record of the arc $\beta$}.

We achieve a significant speedup by doing a preliminary run on each arc, recording the number of positive and negative crossings. We discard the weight set if a proper multicurve is ever detected, which by Corollary \ref{cor:interesectionCount} is the case if and only if the number of crossings is less than $\frac{w_0}{2} + \frac{w_1}{2} - \min(w_1, w_8)$. During the preliminary run, we also use a static array of size $2^m$ to compute the Burau polynomial where the powers of $t$ are treated modulo $2^m$. On our system, $m = 7$ appears to be optimal. When we are solely searching for the zero polynomial, we short-circuit the computation if the array is not all zero.

 \section{Analysis}
  For an arc $\beta$ belonging to a Burau pair $(\alpha, \beta)$, we denote by $\vec{w}_\beta$ the corresponding weights on $\mu$, $c(\beta)$ the number of crossings with $\alpha$, and $p_\beta(t)$ the resulting Burau polynomial. For each $k \in \N$, \emph{the arcs of level $2k$} are $\set{\beta : c(\beta) = 2k}$. Our implementation has resulted in the following theorem.

\begin{theorem}
 For any arc $\beta$ with $c(\beta) \leq 2000$, $p_\beta(t) \neq 0$.
\end{theorem}

Finding no arc-pair asserting the unfaithfulness of $\Psi_4$, we proceed with a more qualitative analysis. Define the minimum norm and multiplicity of a level $2k$ to be
\begin{align*}
 \minnorm(2k) &= \min_{\beta} \set{|p_\beta(t)| : c(\beta) = 2k}\\
 \mult(2k) &= \#\set{\beta : |p_\beta(t)| = \minnorm(2k)}.
\end{align*}
The initial levels have erratic behavior.
\begin{center}
 \begin{tabular}{c|c|ccc|c|ccc|c|c}
  $2k$ & $\minnorm(2k)$ & $\mult(2k)$ & \  & $2k$ & $\minnorm(2k)$ & $\mult(2k)$ & \  & $2k$ & $\minnorm(2k)$ & $\mult(2k)$\\
  \cline{1-3} \cline{5-7} \cline{9-11}
  2 & 2 & 5 && 16 & 8 & 24 && 30 & 8 & 6\\
  4 & 4 & 37 && 18 & 8 & 6 && 32 & 8 & 12\\
  6 & 6 & 115 && 20 & 8 & 24 && 34 & 8 & 18\\
  8 & 6 & 20 && 22 & 6 & 6 && 36 & 10 & 36\\
  10 & 8 & 86 && 24 & 8 & 18 && 38 & 8 & 6\\
  12 & 8 & 30 && 26 & 8 & 12 && 40 & 8 & 18\\
  14 & 8 & 18 && 28 & 8 & 24 && 42 & 10 & 24\\
 \end{tabular}
\end{center}

However after this appears remarkable periodicity.
\begin{theorem}
 \label{thm:periodicity}
 For $2k \in [44, 500]$,
 $$
  \minnorm(2k) = \left\{\begin{array}{cc}10 & \text{if } 2k \equiv 0 \mod{6}\\ 8   & \text{otherwise} \end{array}\right. \qquad \mult(2k) = \left\{\begin{array}{cc}6 & \text{if } 2k \equiv 2 \mod{6}\\ 18 & \text{otherwise} \end{array}\right..
 $$
\end{theorem}
The periodic behavior in Theorem \ref{thm:periodicity} is a result of the arcs falling into into families. 
\begin{example}
 \label{exmample:minnormfamily}
 Let $2k \equiv 0 \mod{6}$ and $44 \leq 2k  \leq 500$. Setting $m = \frac{2k - 48}{6}$, the weights 
  $$w_0 = 22 + 2m,\ w_1 = 74 + 10m,\ w_2 = 89 + 12m,\ w_3 = 21 + 2m,\ w_{14} = 192 + 24m$$
 give rise to arcs $\set{\beta_m}_{m=0}^{158}$ with $c(\beta_m) = 2k$, $|p_{\beta_m}(t)| = \minnorm(2k)$, and
  $$p_{\beta_m}(t) = -t^{-5} - t^{-3} - t^{-2} - t^{0} - t^{18+3m} + t^{19+3m} - 2t^{20+3m} + t^{21+3m} - t^{22+3m}.$$
\end{example}

Our computations have shown that all arcs $\beta$ with $c(\beta) \in [44, 500]$ and $|p_\beta(t)| = \minnorm(c(\beta))$ fall into $6+18+18$ families whose weights are in an arithmetic progression and whose ordered set of coefficients of the resulting Burau polynomials are constant.

There are a variety of ways that these families arise, the simplest is as follows. Let $\beta$ be an arc, $p$ a point on $\beta$, and $\gamma$ a loop that only intersects itself or $\beta$ at $p$. Then with some perturbations one can construct new arcs by repeatedly inserting $\gamma$ into $\beta$. An example of this form is seen in Figure \ref{fig:exampleSimpleFamily}. Note that doing this does not guarantee that the resulting Burau polynomial norms will all be equal, nor that the resulting polynomials themselves will all be of a similar form. More complicated patterns also give rise to families, as seen in Figure \ref{fig:exampleComplexFamily} where there are two loops inserted in an interleaving fashion.

\begin{figure}[!h]
 \centering
 \includegraphics[width=.95\linewidth]{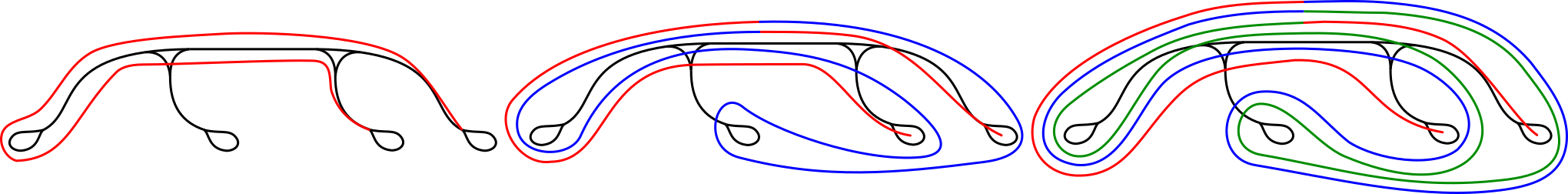}
 \captionof{figure}{Simple family of arcs.}
 \label{fig:exampleSimpleFamily}
\end{figure}

\begin{figure}[!h]
 \centering
 \includegraphics[width=.95\linewidth]{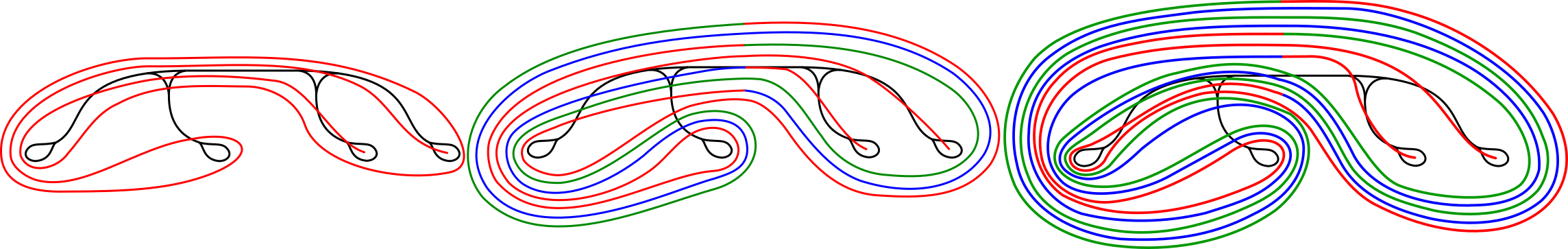}
 \captionof{figure}{Complex family of arcs.}
 \label{fig:exampleComplexFamily}
\end{figure}

The families in Figure \ref{fig:exampleSimpleFamily} and Figure \ref{fig:exampleComplexFamily} are not one of the $6+18+18$ families in Theorem \ref{thm:periodicity}. Indeed, the families in Theorem \ref{thm:periodicity} would be quite difficult to draw. In order to see their behavior, we can look at records of the arcs. Consider Example \ref{exmample:minnormfamily} above, but now with $m = \frac{2k - 6}{6}$ where $2k \equiv 0 \mod{6}$ and $2k \in [6, 500]$. The records of the first three members of this family are as follows.
 $$(3, 5, 3, 5, 4, 2, 4, 2, 4, 2, 4)$$
 $$\left(3, 5, 3, 5, \fbox{4, 1, 2, 4, 2, 4, 2, 5, 3, 5, 3, 5}, 4, 2, 4, 2, 4, 2, 4\right)$$
 $$\left(3, 5, 3, 5, \fbox{4, 1, 2, 4, 2, 4, 2, 5, 3, 5, 3, 5}, \fbox{4, 1, 2, 4, 2, 4, 2, 5, 3, 5, 3, 5}, 4, 2, 4, 2, 4, 2, 4\right)$$
The $m$th member is achieved by inserting the arc with record $4, 1, 2, 4, 
2, 4, 2, 5, 3, 5, 3, 5$ $m$-times. Another family of arcs that gives rise to the minimum norms of Theorem \ref{thm:periodicity} has the following records.
 $$(3, 5, 3, 5, 3, 5, 4, 2, 4, 2, 4)$$
 $$\left(\fbox{3, 5, 3, 5, 3}, \fbox{6, 5, 4, 2, 4, 2, 4}, 3, 5, 3, 5, 3, 5, 4, 2, 4, 2, 4\right)$$
 $$\left(\fbox{3, 5, 3, 5, 3}, \fbox{6, 5, 4, 2, 4, 2, 4}, \fbox{3, 5, 3, 5, 3}, \fbox{6, 5, 4, 2, 4, 2, 4}, 3, 5, 3, 5, 3, 5, 4, 2, 4, 2, 4\right)$$
Here the arcs $3, 5, 3, 5, 3$ and $6, 5, 4, 2, 4, 2, 4$ are interleaved. In each of these examples, the initial members of the family do not give rise to Burau polynomials with the same ordered set of coefficients. It is only the later members that have this behavior.

\/*\begin{center}
 \includegraphics[scale=0.25]{Family-simple01.png}
 \includegraphics[scale=0.25]{Family-simple02.png}
 \includegraphics[scale=0.25]{Family-simple03.png}
\end{center}

More complicated patterns also arise.

\begin{center}
 \includegraphics[scale=0.3]{Family-complex01.png}
 \includegraphics[scale=0.3]{Family-complex02.png}
 \includegraphics[scale=0.3]{Family-complex03.png}
\end{center}

In the family above, each additional step inserts two loops; one blue and one green.  Due to the way they are interleaved, the number of blue loops must equal that of green loops. 

While neither of these two families give rise to arcs that achieve minimum norm, the behavior of those that do is similar. We can see this by looking at the logs of the arcs. The first family pictured above has the following logs.
 $$(4)$$
 $$(4, 1, 2, 5, 4)$$
 $$(4, 1, 2, 5, 4, 1, 2, 5, 4)$$
 $$(4, 1, 2, 5, 4, 1, 2, 5, 4, 1, 2, 5, 4)$$
The insertion of the blue loop can be seen by the insertion of $1, 2, 5, 4$ in the log. The second family pictured above has the following logs.
 $$(5, 6, 7)$$
 $$(6, 1, 5, 6, 7, 6, 7)$$
 $$(6, 1, 6, 1, 5, 6, 7, 6, 7, 6, 7)$$
 $$(6, 1, 6, 1, 6, 1, 5, 6, 7, 6, 7, 6, 7, 6, 7)$$
The red arc is $5, 6, 7$, the blue arc is $6, 1$, and the green arc is $6, 7$.

The same behavior occurs in all of the minimum norm weights mentioned above. For example, consider the family ${8 + 2k, 4 + 10k, 5 + 12k, 7 + 2k, 24 + 24k}$ where we now set $k = \frac{c - 6}{6}$ with $c \equiv 0 \mod{6}$ and $c \in [6, 500]$. This is the family of minimum norm discussed above. The logs of the first three members of this family are as follows.
 $$(3, 5, 3, 5, 4, 2, 4, 2, 4, 2, 4)$$
 $$(3, 5, 3, 5, 4, 1, 2, 4, 2, 4, 2, 5, 3, 5, 3, 5, 4, 2, 4, 2, 4, 2, 4)$$
 $$(3, 5, 3, 5, 4, 1, 2, 4, 2, 4, 2, 5, 3, 5, 3, 5, 4, 1, 2, 4, 2, 4, 2, 5, 3, 
5, 3, 5, 4, 2, 4, 2, 4, 2, 4)$$
In this family, the $k$th member is achieved by inserting the arc $4, 1, 2, 4, 
2, 4, 2, 5, 3, 5, 3, 5$ $k$-times. The 
families of minimum norm also exhibit more complex patterns. For example, the weights $(20, 76, 91, 19, 192)$ occurs in a family whose logs are as follows.
 $$(3, 5, 3, 5, 3, 5, 4, 2, 4, 2, 4)$$
 $$(3, 5, 3, 5, 3, 6, 5, 4, 2, 4, 2, 4, 3, 5, 3, 5, 3, 5, 4, 2, 4, 2, 4)$$
 $$(3, 5, 3, 5, 3, 6, 5, 4, 2, 4, 2, 4, 3, 5, 3, 5, 3, 6, 5, 4, 2, 4, 2, 4, 3, 
5, 3, 5, 3, 5, 4, 2, 4, 2, 4)$$
Here the arcs $3, 5, 3, 5, 3$ and $6, 5, 4, 2, 4, 2, 4$ are interleaved.*/

Grouping the arcs into families, the periodicity observed in Theorem \ref{thm:periodicity} is recast into an observation about $6+18+18$ families. These families always being of minimal norm block us from seeing the behavior of other arcs. In order to see through this obstruction, we group all arcs into ``tightly-knit'' families and study the minimum norm where we ignore any noninitial member of such a family.

Let $\sF = \set{\beta_i}_{i=0}^\infty$ be a collection of arcs and set $\vec{d} = \vec{w}_{\beta_1} - \vec{w}_{\beta_0}$. We say that $\sF$ is a \emph{tightly-knit family} if for all $i \in \N$, $\vec{w}_{\beta_i} = \vec{w}_{\beta_0} + i \cdot \vec{d}$ and $p_{\beta_i}(t)$ has the same ordered set of coefficients as $p_{\beta_0}(t)$. We call this the \emph{tightly-knit family generated by $\beta_0$ and $\beta_1$}. For such a family we define $c(\sF) = c(\beta_0)$.

Note that any given arc can lie in many tightly-knit families. Let $\beta$ and $\beta'$ be two distinct arcs with the entries of $\vec{w}_{\beta'} - \vec{w}_{\beta}$ all nonnegative. Then $\beta$ and $\beta'$ lie in a tightly-knit family if and only if $\beta_0 = \beta$ and $\beta_1 = \beta'$ themselves generate a tightly-knit family. To check that $\beta$ and $\beta'$ belong to the same family in computations, we must restrict ourselves to testing the first $m \in \N$ members generated by $\beta$ and $\beta'$. In the results below, we first used $m=100$ and then $m=1000$. In each case, the results remain unchanged.

For $k \in \N$ we define 
 $$\min(2k) = \min_\beta \set{|p_\beta(t)| : c(\beta) = 2k \text{ and } \beta \notin \sF \text{ for any } \sF \text{ with } c(\sF) < 2k}.$$
Note that there exists a positive $k \in \N$ such that $\min(2k) = 0$ if and only if $\Psi_4$ is unfaithful. We emphasize that in computing $\min(2k)$ for $2k \in [2, 200]$, seen in Figure \ref{fig:tightlyKnitFamilies}, we only test the first $100$, and then $1000$, members to determine if a collection forms a tightly-knit family.

\/*\begin{figure}[!h]
  \centering
  \includegraphics[width=6cm]{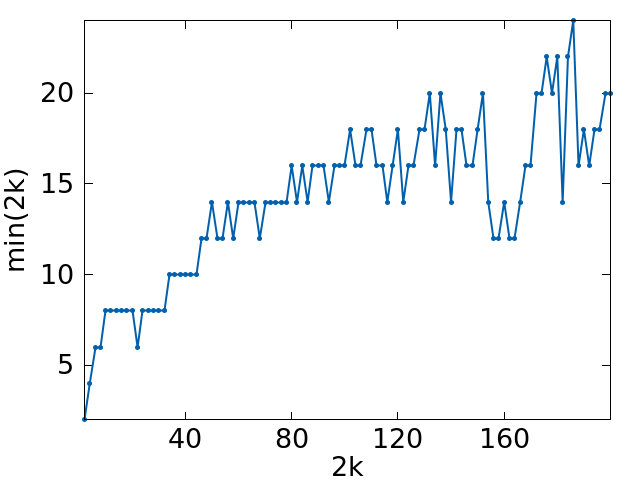}
  \captionof{figure}{Tightly-knit families}
  \label{fig:tightlyKnitFamilies}
\end{figure}*/

By weakening the definition of a tightly-knit family while maintaining the property that there exists a positive $k \in \N$ such that $\min(2k) = 0$ if and only if $\Psi_4$ is unfaithful, we can seek different ways to group the arcs into families, studying their behavior as the number of intersections increase. With this in mind, we say that $\sF = \set{\beta_i}_{i=0}^\infty$ is a \emph{loosely-knit} family if for all $i \in \N$, $\vec{w}_{\beta_i} = \vec{w}_{\beta_0} + i \cdot (\vec{w}_{\beta_1} - \vec{w}_{\beta_0})$ and $|p_{\beta_i}(t)| \geq |p_{\beta_0}(t)|$ for all $i \in \N$. The result of computing $\min(2k)$ for $2k \in [2, 200]$ using loosely-knit families is seen in Figure \ref{fig:looselyKnitFamilies}, again with the computational limitations as mentioned above.

\/*\begin{figure}[!h]
  \centering
  \includegraphics[width=6cm]{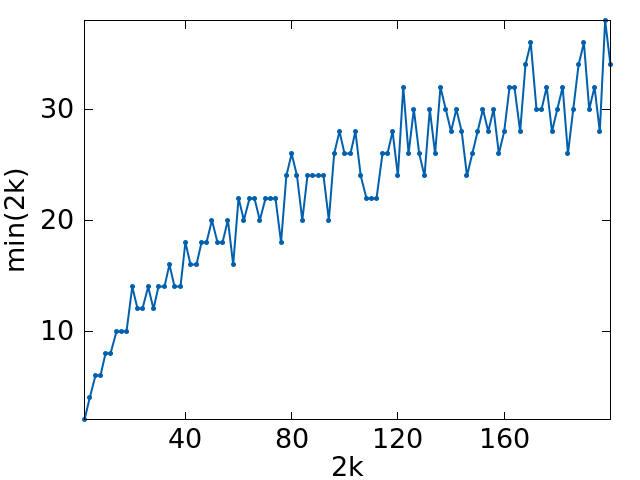}
  \captionof{figure}{Loosely-knit families}
  \label{fig:looselyKnitFamilies}
\end{figure}*/

\begin{figure}[!h]
 \centering
 \begin{subfigure}[t]{0.4\textwidth}
  \centering
  \includegraphics[width=2in]{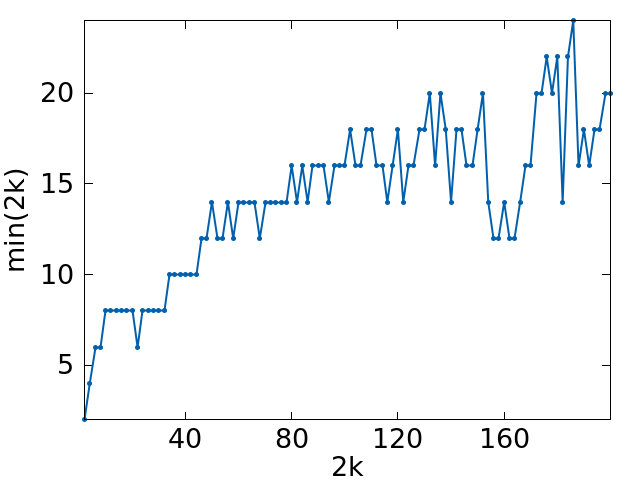}  
  \caption{Tightly-knit families}
  \label{fig:tightlyKnitFamilies}
 \end{subfigure}
 \quad
 \begin{subfigure}[t]{0.4\textwidth}
  \centering
  \includegraphics[width=2in]{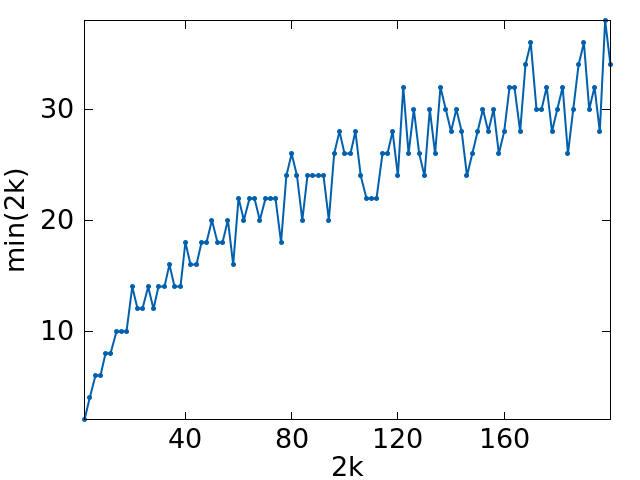} 
  \captionof{figure}{Loosely-knit families}
  \label{fig:looselyKnitFamilies}
 \end{subfigure}
 \caption{}
\end{figure}

We interpret the upward trend of these graphs as being suggestive that $\Psi_4$ is faithful.
  
 \newpage

\bibliographystyle{plain}
\bibliography{burau}

\end{document}